\def\marker{\>\hbox{${\vcenter{\vbox{
\hrule height 0.4pt\hbox{\vrule width 0.4pt height 6pt
\kern6pt\vrule width 0.4pt}\hrule height 0.4pt}}}$}\>}

\documentclass[12pt]{article}
\usepackage{amsmath,amssymb,amsthm}
\usepackage{graphicx}
\usepackage{subfigure}
\usepackage{psfrag}
\usepackage{color}
\usepackage{enumerate}
\usepackage{epstopdf}
\usepackage{indentfirst}
\newtheorem{thm}{Theorem}[section]

\newtheorem{ques}[thm]{Question}

\newtheorem{lem}[thm]{Lemma}
\newtheorem{obs}[thm]{Observation}
\theoremstyle{definition}
\newtheorem{defn}[thm]{Definition}

\def\dfn#1{{\it #1}}

\title{Sharp upper bounds on the $k$-independence number in graphs with given minimum and maximum degree}

\author{
Suil O\footnotemark[1],\quad
Yongtang Shi\footnotemark[2],\quad
Zhenyu Taoqiu\footnotemark[2]\quad
}
\date{\today}

\begin{document}
\maketitle
\begin{abstract}
The \dfn{$k$-independence number} of a graph $G$ is the maximum size of a set of vertices at pairwise distance greater than $k$. In this paper, for each positive integer $k$, we prove sharp upper bounds for the $k$-independence number in an $n$-vertex connected graph with given minimum and maximum degree.
\\


\end{abstract}

\renewcommand{\thefootnote}{\fnsymbol{footnote}}
\footnotetext[1]{
Department of Applied Mathematics and Statistics,
The State University of New York, Korea, Incheon, 21985;
suil.o@sunykorea.ac.kr.
Research supported by NRF-2017R1D1A1B03031758 and by NRF-2018K2A9A2A06020345.}
\footnotetext[2]{
Center for Combinatorics and LPMC,
Nankai University, Tianjin 300071, China;
tochy@mail.nankai.edu.cn (corresponding author), shi@nankai.edu.cn.  Research supported by the National
Natural Science Foundation of China (No. 11811540390).}

\section{Introduction}

\baselineskip 17pt
Throughout this paper, all graphs are simple, undirected, and finite. For two vertices $u$ and $v$ in a graph $G$, we define the distance between $u$ and $v$, written $d_G(u,v)$ or simply $d(u,v)$, to be the length of the shortest path between $u$ and $v$.
For a nonnegative integer $k$, a \textit{$k$-independent set} in a graph $G$ is a vertex set $S \subseteq V(G)$ such that
the distance between any two vertices in $S$ is bigger than $k$. Note that the 0-independent set is $V(G)$ and an 1-independent set is an independent set.
The \textit{$k$-independence number} of a graph $G$, written $\alpha_k(G)$, is the maximum size of a $k$-independent set in G.

It is known that $\alpha_1(G)=\alpha(G) \ge \frac n{\chi(G)}$, where $\chi(G)$ and $\alpha(G)$ are the chromatic number and independence number of a graph $G$, repsectively.
Similarly, by finding the $k$-distance chromatic number of $G$, we can find a lower bound for $\alpha_k(G)$. It will be discussed in Section 4.
Other graph parameters such as the average distance~\cite{FH}, injective chromatic number~\cite{HKSS}, packing chromatic number~\cite{GHHHR}, and strong chromatic index~\cite{M} are
also directly related to the $k$-independence number. Lower bounds on the corresponding distance or packing chromatic
number can be given by finding upper bounds on the k-independence number. Alon and Mohar~\cite{AM} asked the extremal value for the distance chromatic number in graphs of a given girth and degree.

Firby and Haviland~\cite{FH} proved an upper bound for $\alpha_k(G)$ in an $n$-vertex connected graph. We give a proof of the theorem below, because with a similar idea, we prove Theorem~\ref{main2}, which is one of the main results in this paper.
\begin{thm}{\rm(~\cite{FH})} \label{thm1}
	For a positive integer $k$, if $G$ is a non-complete $n$-vertex connected graph with $diam(G)\geq k+1$, then
	$$2\leq \alpha_k(G)\leq \left\{ \begin{aligned}
	\ &\frac{2n}{k+2} ,&\text{ if }  k \text{ is even, }\\
	\ &\frac{2n-2}{k+1} ,&\text{ if }  k \text{ is odd.}
	\end{aligned}  \right.$$
	Furthermore, bounds  are sharp.
\end{thm}
\begin{proof}
	Let $S_k$ be a $k$-independent set in $G$.
	Since $diam(G)\geq k+1$ and $G$ is not a complete graph, there are two vertices $u,v \in V(G)$ such that $d_G(u,v)=k+1$. Thus $u,v \in S_k$, which implies $\alpha_k(G)= max \ |S_k| \geq 2$. The graph $K_1 \vee K_{i_1} \vee \cdots K_{i_k}\vee K_1$ attains equality in the lower bound, where $\sum_{j=1}^k i_j=n-2$.\\
	For the upper bounds, we consider two cases depending on the parity of $k$.\\
	\textit{Case 1:} $k$ is even. 
	When $k=2$, for any pair of vertices $u,v \in S_2$, we have $d_G(u,v) \geq 3$, which means $N(u)\cap N(v)=\emptyset$. Therefore, we have $|N(S_2)| \geq |S_2|$. Since $|N(S_2)|+|S_2|=n$, we have $|S_2|\leq \frac{n}{2}$. The $n$-vertex comb $H$ has $\alpha_2(H)=\frac{n}{2}$, where a comb is a graph obtained by joining a single pendant edge to each vertex to a path. For $k \ge 4$ and any pair of vertices $u,v \in S_k$, we have $d_G(u,v) \geq k+1$ and $N(v)\cap N(u)=\emptyset$. Thus we have $|N(S_k)| \geq |S_k|$. Simliarly, for $j=1,\ldots,\frac k2$, we have $|N^j(S_k)| \ge |N^{j-1}(S_k)|$. Thus we have $n-|S_k| - \frac k2 |N(S_k)|\ge 0$, which implies $\alpha_k(G)=max \ |S_k| \leq \frac{2n}{k+2}$. The $n$-vertex graph $H_k$ obtained from a comb $H$ with $\frac{4n}{k+2}$ vertices by replacing each pendant edge of $H$ with a path of length $\frac{k}{2}$ has $\alpha_k(H_k)=\frac{2n}{k+2}$.\\
	\textit{Case 2:} $k$ is odd. When $k=1$, for any pair of vertices $u,v \in S_1$, we have $d_G(u,v)\geq 2$. Thus we have $n-|S_1|\geq 1$. The star $K_{1,n-1}$ have $\alpha_1(G)=n-1$. For $k \ge 3$, for any pair of vertices $u,v \in S_k$, we have $d_G(u,v) \geq k+1$ and $N(v)\cap N(u)=\emptyset$. Similarly to Case 1, we have $|N(S_k)| \geq |S_k|$ and for $j=1,\ldots,\frac {k-1}2$, we have $|N^j(S_k)| \ge |N^{j-1}(S_k)|$ and $N^{\frac{k+1}2}(S_k) \neq \emptyset$. Thus we have $n-|S_k|- \frac {k-1}2 |N(S_k)| \ge 1$, which implies $\alpha_k(G)=max \ |S_k| \leq \frac{2(n-1)}{k+1}$. The graph $F_k$ obtained from the star $K_{1, \frac{2(n-1)}{k+1}}$ by replacing each edge with a path of length $\frac{k+1}{2}$ has $\alpha_k(F_k)=\frac{2n-2}{k+1}$.
\end{proof}

For a vertex set $S \subseteq V(G)$, let $N(S)$ be the neiborhood of $S$, and for an integer $j\ge 2$, let $N^j(S)=N(N^{j-1}(S))\setminus (N^{j-2}(S)\cup N^{j-1}(S))$, where $N^0(S)=S$ and $N^1(S)=N(S)$. For graphs $G_1,\ldots,G_k$, the graph $G_1\vee \cdots \vee G_k$ is the one such that $V(G_1\vee \cdots \vee G_k)$ is the disjoint union of $V(G_1), \ldots, V(G_k)$ and $E(G_1\vee \cdots \vee G_k)=\{e: e \in E(G_i) \text{ for some } i\in [k] \text{ or an unordered pair between } V(G_i) \text{ and } V(G_{i+1}) \text{ for some } i\in [k-1]\}$.

In 2000, Kong and Zhao~\cite{KZ1} showed that for every $k \ge 2$, determining $\alpha_k(G)$ is NP-complete for general graphs. They also showed that this problem remains NP-hard for regular bipartite graphs when $k \in \{2, 3, 4\}$~\cite{KZ2}. It is well-known that for an $n$-vertex $r$-regular graph $G$, we have $\alpha_1(G) \le \frac n2$. Also, for $k=2$, we have $\alpha_2(G) \le \frac n{r + 1}$ because for any pair of two vertices $u, v$ in a 2-independent set, we have $N(u) \cap N(v) = \emptyset$, which implies $n \ge |S_2|+|N(S_2)|\ge |S_2|+r|S_2|$.
For each fixed integer, $k \ge 2$ and $r \ge 3$, Beis, Duckworth, and Zito~\cite{BDZ} proved some upper bounds for $\alpha_k(G)$ in random $r$-regular graphs.

The remainder of the paper is organized as follows. In Subsection 2, for all positive integers $k$ and $r \ge 3$, we provide infinitely many $r$-regular graphs with $\alpha_k(G)$ attaining the sharp upper bounds. In Section 3, we prove sharp upper bounds  for $\alpha_k(G)$ in an $n$-vertex connected graph with $diam(G)\ge k+1$ for every positive integer $k$ with given minimum and maximum degree. We conclude this paper with some open questions in Section 4.

For undefined terms, see West~\cite{W}.

\section{Construction}
In this section, we construct $n$-vertex $r$-regular graphs with the $k$-independence number achieving equality in the upper bounds in Theorem~\ref{main2}. For a vertex $v \in V(G)$, we denote the neighborhood of $v$ by $N(v)$ and $N(v) \cup \{v\}$ by $N[v]$, respectively.

\begin{defn} \label{def1}
For a positive integer $\ell$, let $k=6\ell-4$. Let $H_{r,k}^1$ be the $r$-regular graph with the vertex sets $V_1,\ldots,V_{3\ell-1}$ satisfying the following properties:\\
(i) $V_1$ is an independent set with $r$ vertices $v_{11},\cdots,v_{1r}$ such that for each $i\in [r]$, the degree of $v_{1i}$ is $r$, $N(v_{1i})$ induces a copy of $K_r-K_2$ and $N(v_{1i})\cap N(v_{1j})=\emptyset$ for $j \neq i$.\\
 (ii) Let $V_2=\cup_{j=1}^r N(v_{1j})$ such that for each $i\neq j\in [r]$, there is no edge with endpoints in $N(v_{1i})$ and $N(v_{1j})$, and for each $i\in [r]$, $v^1_{2i}, v^2_{2i}\in N(v_{1i})$, and $v^1_{2i}$ is not adjacent to $v^2_{2i}$. \\
(iii) For a positive integer $x\in[\ell-1]$, let $V_{3x}=\{v_{(3x)1},\cdots,v_{(3x)r} \}$ such that for each $i\in [r]$, $v_{(3x)i}$ is adjacent to $v^h_{(3x-1)i}$ for $h\in \{1,2\}$, $N(v_{(3x)i})\setminus v^h_{(3x-1)i}$ induces a copy of $K_{r-2}$ (in $V_{3x+1}$), and for each $i\neq j\in [r]$, $N[v_{(3x)i}]\cap N[v_{(3x)j}]=\emptyset$.\\
(iv) Let $V_{3x+1}=\{N(v_{(3x)1})\setminus v^h_{(3x-1)1},\ldots,N(v_{(3x)r})\setminus v^h_{(3x-1)r} \}$ such that $h \in \{1,2\}$ and  for each $i\neq j \in [r]$, there is no edge with endpoints in $N(v_{(3x)i})\setminus v_{(3x-1)i}$ and in $N(v_{(3x)j})\setminus v_{(3x-1)j} $.\\
(v) Let $V_{3x+2}=\{v^1_{(3x+2)1},v^2_{(3x+2)1},\ldots,v^1_{(3x+2)r},v^2_{(3x+2)r} \}$ such that for each $i\in [r]$, $v^1_{(3x+2)i}$ is adjacent to $v^2_{(3x+2)i}$, and $v^h_{(3x+2)i}$ is adjacent to all vertices in $N(v_{(3x)i})\setminus v_{(3x-1)i}$, for each $i\neq j\in [r]$ and $h \in \{1,2\}$, $v^h_{(3x+2)i}$ is not adjacent to $v^h_{(3x+2)j}$ except for $x=\ell-1$.\\
Let $G^1_{r,k,t}$ be the disjoint union of $t$ copies of $H^1_{r,k}$ (see Figure~\ref{fig:1}).
\end{defn}

\begin{figure}[!ht]\centering
	\includegraphics[scale=0.25]{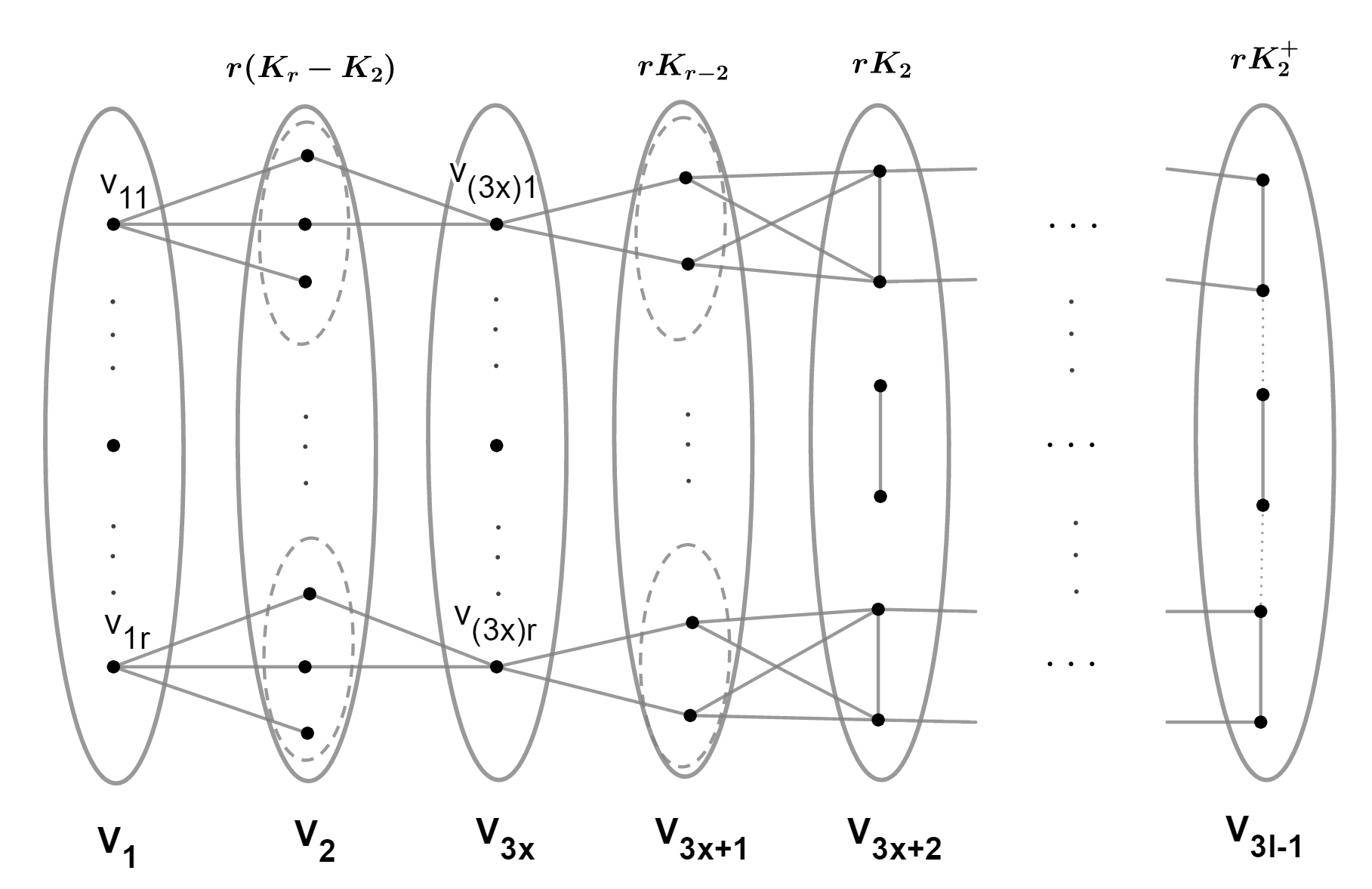}
	\caption{The graph $H^1_{r,k}$ \label{fig:1}}
\end{figure}

\begin{obs} \label{obs1}
	The graph $G^1_{r,k,t}$ in Definition~\ref{def1} is an $r$-regular graph with $n=t\ell r(r+1)$ vertices and the $k$-independence number $\frac{n}{\ell (r+1)}$.
\end{obs}

\begin{proof}
By the definition of $G^1_{r,k,t}$, every vertex has degree $r$, $V_1$ is a $k$-independent set with size $\frac{n}{\ell (r+1)}$, and for any $u,v \in V_1$ and for any $x, y \in V(G^1_{r,k,t})$, we have $k+1=d(u,v) \ge d(x,y)$. Also, we have $\sum_{i=1}^{3\ell-1}|V_i|=t\ell r(r+1)$, which gives the desired result.
\end{proof}

For $k=6\ell-4$, we can create other $r$-regular graphs with the $k$-independence number equal to $\frac{n}{\ell (r+1)}$ (see Figure~\ref{fig:2} and~\ref{fig:3}).

\begin{figure}[!ht]\centering
\includegraphics[scale=0.25]{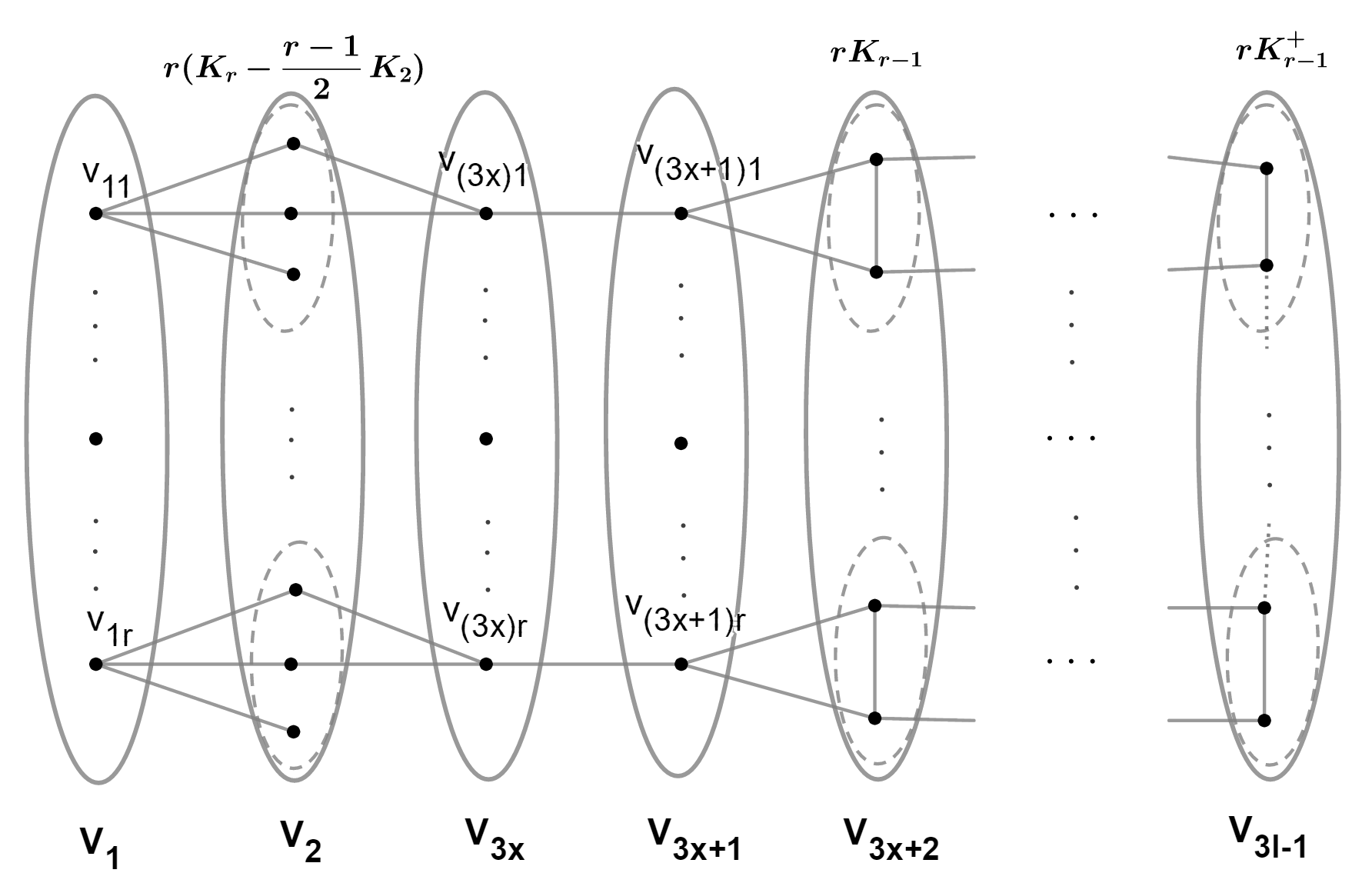}
\caption{The graph $H^2_{r,k}$ \label{fig:2}}
\end{figure}

\begin{figure}[!ht]\centering
\includegraphics[scale=0.25]{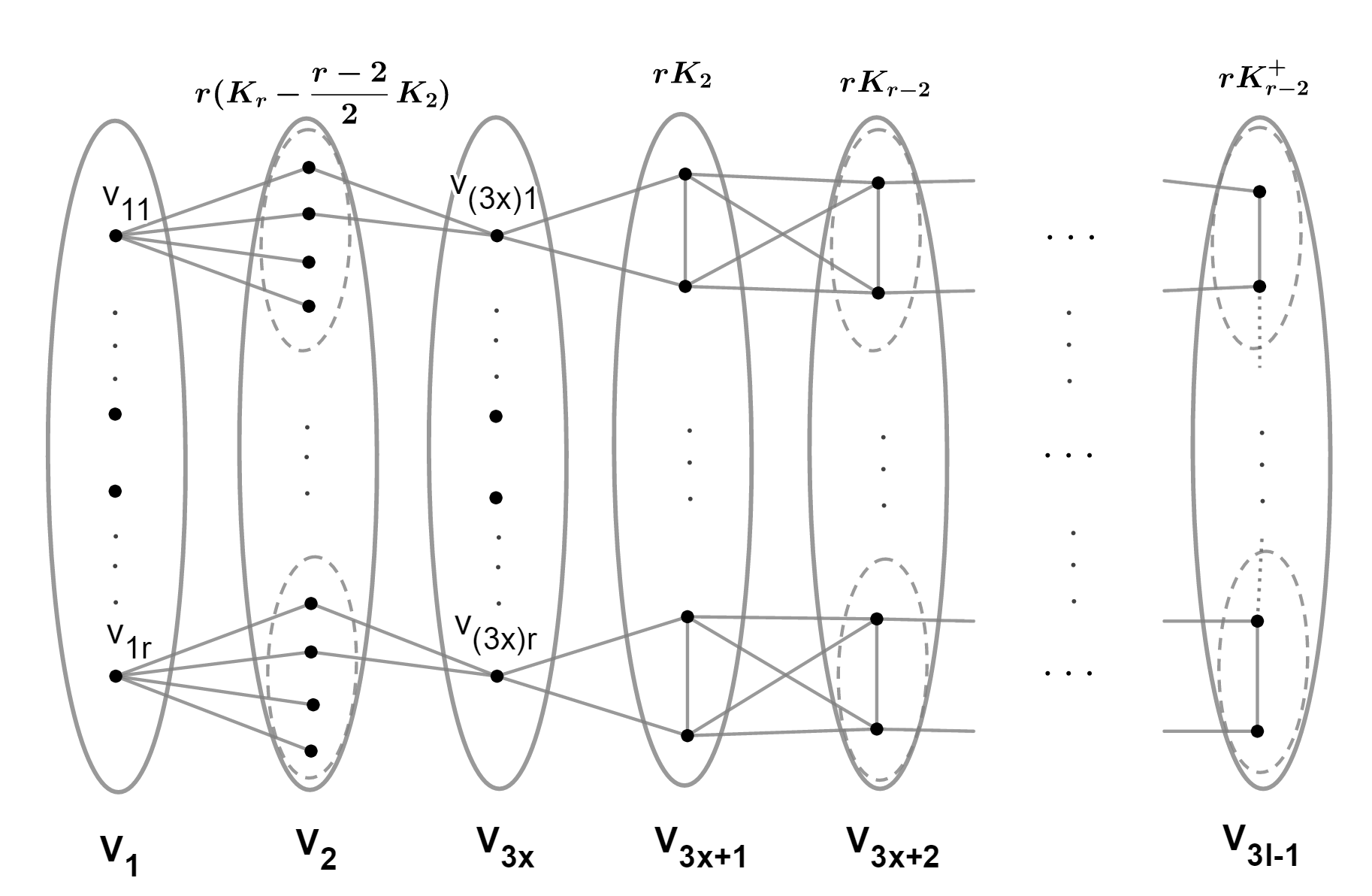}
\caption{The graph $H^3_{r,k}$ \label{fig:3}}
\end{figure}

\begin{defn} \label{def2}
For a positive integer $\ell$, let $k=6\ell-3$. Let $H_{r,k}^4$ be the $r$-regular graph with the vertex sets $V_1,\ldots,V_{3\ell}$ satisfying the following properties:\\
(i) For $x \in [\ell-1]$, follow the definitions of $V_1,\ V_2,\ V_{3x},\ V_{3x+1},\ V_{3x+2}$ in Definition~\ref{def1}, and in $V_{3\ell-1}$, for each $i\neq j\in [r]$ and for $h \in \{1,2\}$, $v^h_{(3\ell-1)i}$ is not adjacent to $v^h_{(3\ell-1)j}$. (Note that $V_{3\ell-1}$ in this definition is different from the one in Definition~\ref{def1}).\\
(ii) Let $V_{3\ell}=\{v_{(3\ell)1},v_{(3\ell)2} \}$ such that for each $i\in [r]$, $v_{(3\ell)1}$ is adjacent to $v^1_{(3\ell-1)i}$ and $v_{(3\ell)2}$ is adjacent to $v^2_{(3\ell-1)i}$. \\
Let $G^4_{r,k,t}$ be the disjoint union of $t$ copies of $H^4_{r,k}$ (see Figure~\ref{fig:4}).
\end{defn}

\begin{figure}[!ht]\centering
\includegraphics[scale=0.25]{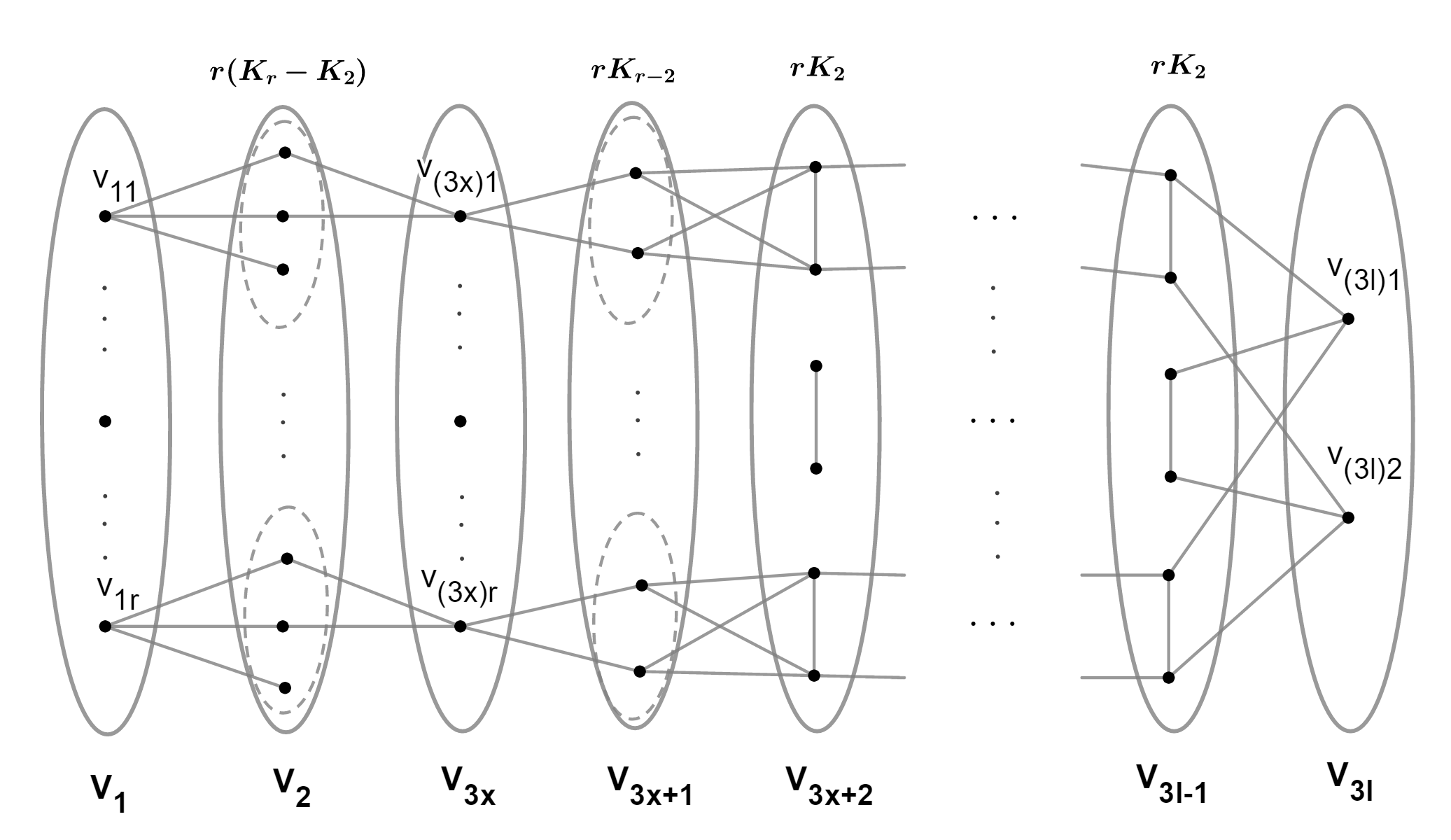}
\caption{The graph $H^4_{r,k}$ \label{fig:4}}
\end{figure}

Similarly to Observation~\ref{obs1}, Definition~\ref{def2} guarantees the following observation.

\begin{obs}
The graph $G^4_{r,k,t}$ in Definition~\ref{def2} is an $r$-regular graph with $n=t\ell r(r+1)+2t$ vertices and the $k$-independence number $\frac{rn}{\ell r(r+1)+2}$.
\end{obs}

\begin{defn} \label{def3}
For a positive integer $\ell$, let $k=6\ell-2$. Let $H_{r,k}^5$ be the $r$-regular graph with the vertex sets $V_1,\ldots,V_{3\ell}$ satisfying the following properties:\\
(i) For $x \in [\ell-1]$, follow the definitions of $V_1,\ V_2,\ V_{3x}, \ V_{3x+1}, \ V_{3x+2}$ in Definition~\ref{def2} except $|V_1|=r-1$.\\
(ii) In $V_{3\ell}$, $v_{(3\ell)i}$ is adjacent to $v_{(3\ell)j}$ for $i \neq j \in [r-1]$, i.e., the graph induced by $V_{3\ell}$ is a copy of $K_{r-1}$.\\
Let $G^5_{r,k,t}$ be the disjoint union of $t$ copies of $H^5_{r,k}$ (see Figure~\ref{fig:5}).
\end{defn}

\begin{figure}[!ht]\centering
\includegraphics[scale=0.2]{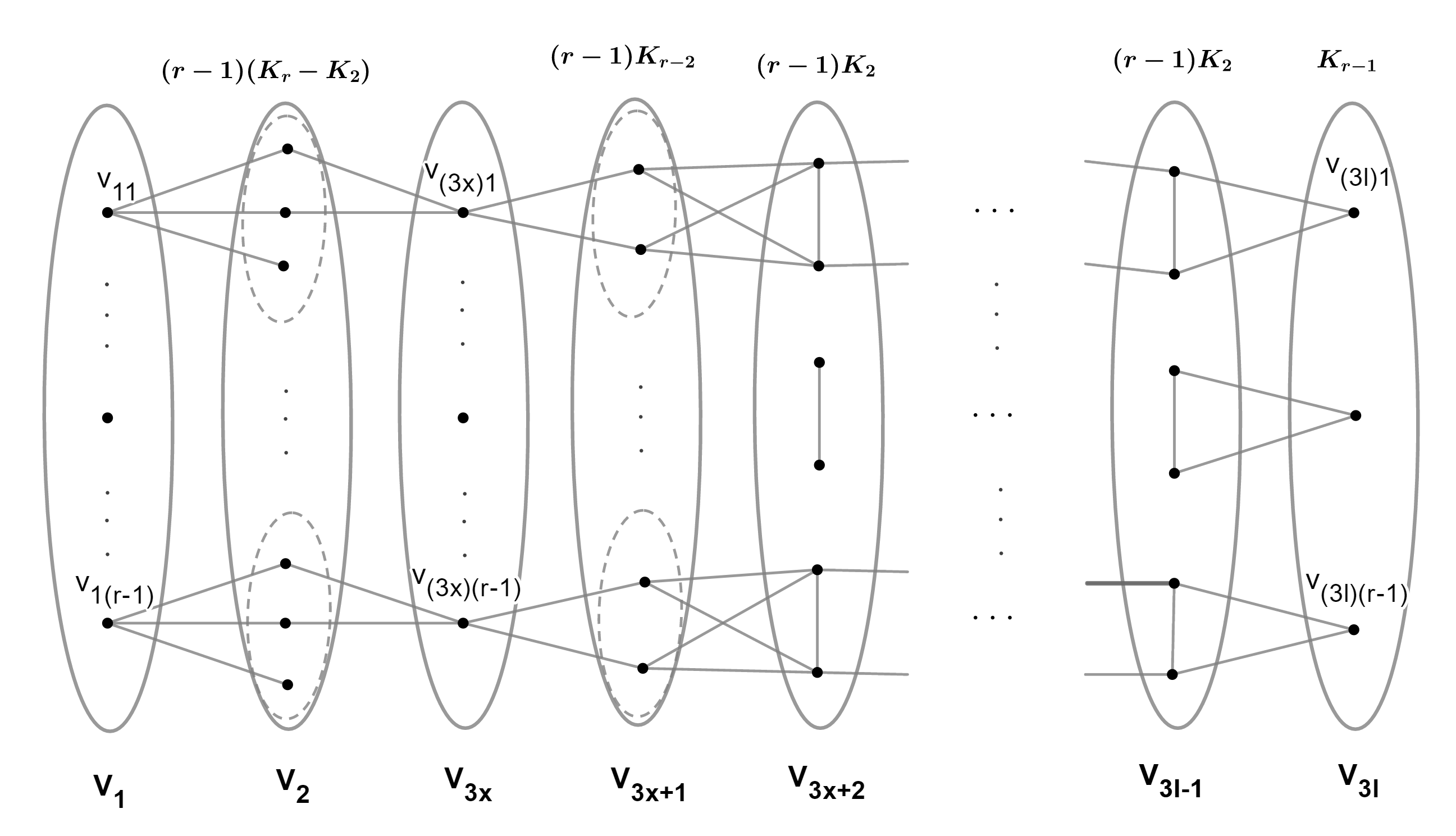}
\caption{The graph $H^5_{r,k}$ \label{fig:5}}
\end{figure}

\begin{obs} \label{obs3}
The graph $G^5_{r,k,t}$ in Definition~\ref{def3} is an $r$-regular graph with $n=t\ell (r-1)(r+1)+t(r-1)$ vertices and the $k$-independence number $\frac{n}{\ell (r+1)+1}$.
\end{obs}

\begin{defn} \label{def4}
Let $r$ be an odd interger at least $3$, and for a positive integer $\ell$, let $k=6\ell-1$. Let $H^6_{r,k}$ be the $r$-regular graph with the vertex sets $V_1,\ldots,V_{3\ell+1}$ satisfying the following properties:\\
(i) For $x\in [\ell]$, follow the definitions of $V_1,\ V_2,\ V_{3x}, \ V_{3x+1}, \ V_{3x+2}$ in graph $H_{r,k}^2$ (see Figure~\ref{fig:2}), except $V_{3\ell+1}$.(Note that $V_{3\ell-1}$ in this definition is different from the one in $H_{r,k}^2$).\\
(ii) Let $|V_{3\ell+1}|=1$ such that all vertices in $V_{3\ell}$ are adjacent to the vertex in $V_{3\ell+1}$.\\
Let $G^6_{r,k,t}$ be the disjoint union of $t$ copies of $H^6_{r,k}$ (see Figure~\ref{fig:6}).
\end{defn}

\begin{figure}[!ht]\centering
\includegraphics[scale=0.2]{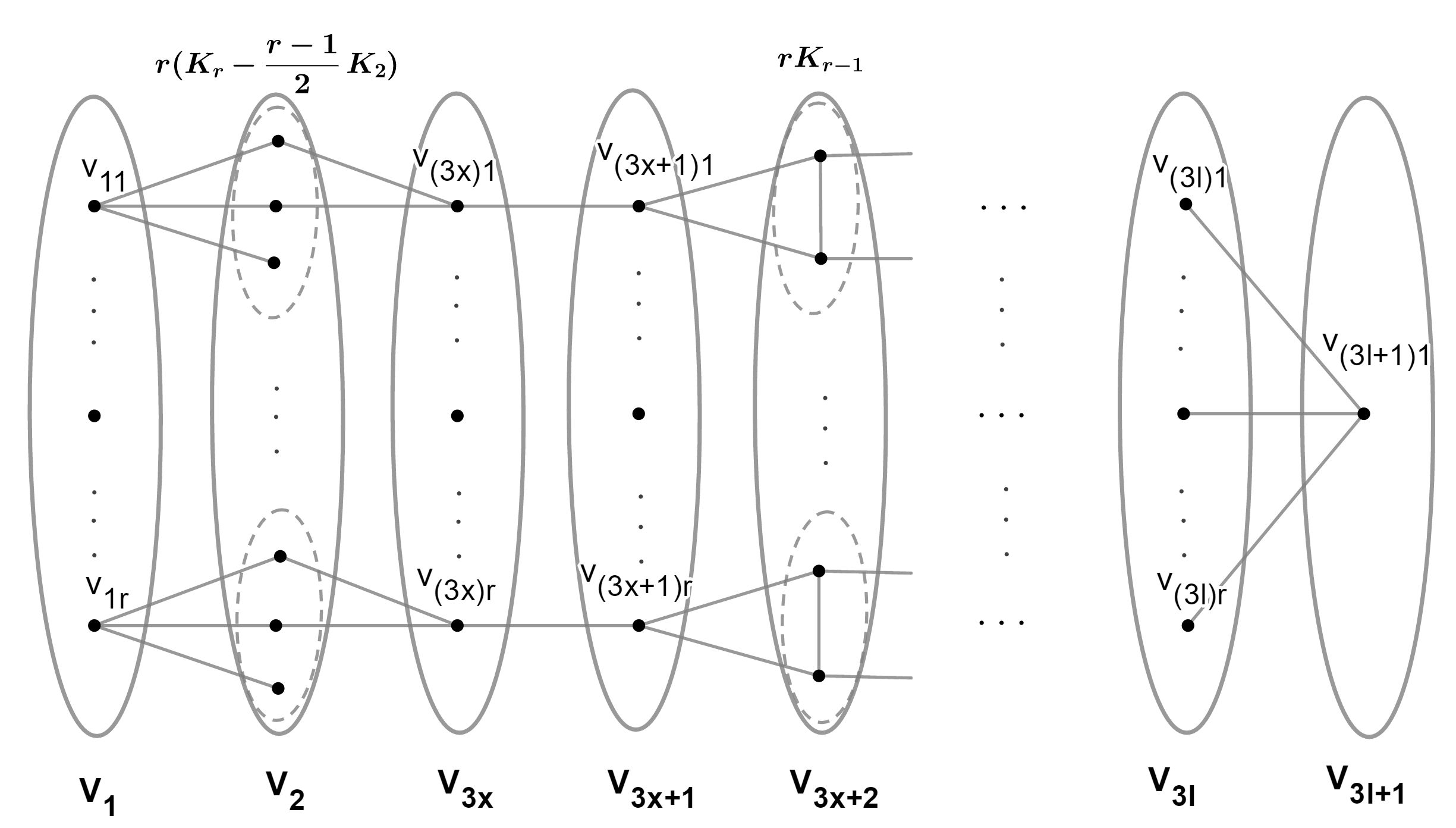}
\caption{The graph $H^6_{r,k}$ \label{fig:6}}
\end{figure}

\begin{defn} \label{def5}
Let $r$ be an even interger at least $4$, and for a positive integer $\ell$, let $k=6\ell-1$. Let $H^7_{r,k}$ be the $r$-regular graph with the vertex sets $V_1,\ldots,V_{3\ell+1}$ satisfying the following properties:\\
(i) For $x\in [\ell]$, follow the definitions of $V_1,\ V_2,\ V_{3x}, \ V_{3x+1}, \ V_{3x+2}$ in graph $H_{r,k}^3$ (see Figure~\ref{fig:3}), except $V_{3\ell+1}$.(Note that $V_{3\ell-1}$ in this definition is different from the one in $H_{r,k}^3$).\\
(ii) Let $|V_{3\ell+1}|=2$ such that all vertices in $V_{3\ell}$ are adjacent to the two vertices in $V_{3\ell+1}$, and $V_{3\ell+1}$ is independent.\\
Let $G^7_{r,k,t}$ be the disjoint union of $t$ copies of $H^7_{r,k}$ (see Figure~\ref{fig:7}).
\end{defn}

\begin{figure}[!ht]\centering
\includegraphics[scale=0.25]{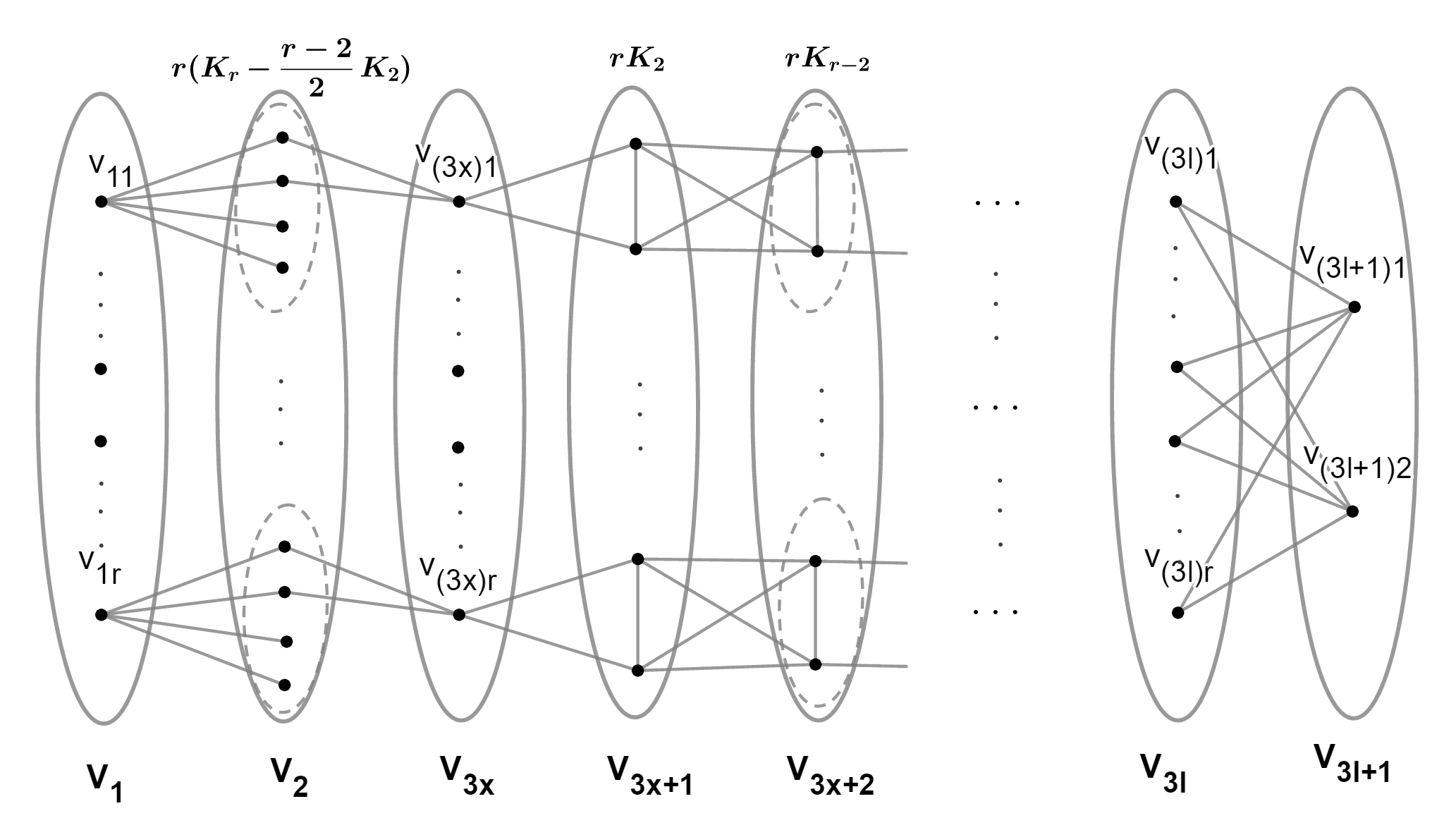}
\caption{The graph $H^7_{r,k}$ \label{fig:7}}
\end{figure}

\begin{obs} \label{obs4}
The graph $G^6_{r,k,t}$ in Definition~\ref{def4} is an $r$-regular graph with $n=t(\ell r+1)(r+1)$ vertices and the $k$-independence number $\frac{rn}{(\ell r+1)(r+1)}$.\\
Also, the graph $G^7_{r,k,t}$ in Definition~\ref{def5} is an $r$-regular graph with $n=t(\ell r+1)(r+1)+t$ vertices and the $k$-independence number $\frac{rn}{(\ell r+1)(r+1)+1}$.
\end{obs}

\begin{defn} \label{def6}
Let $r$ be an odd integer at least $3$, and for a positive integer $\ell$, let $k=6\ell$. Let $H^8_{r,k}$ be the $r$-regular graph with the vertex sets $V_1,\ldots,V_{3\ell+1}$ satisfying the following properties: \\
(i) For $x\in [\ell]$, follow the definitions of $V_1,\ V_2,\ V_{3x}, \ V_{3x+1}, \ V_{3x+2}$ in Definition~\ref{def4}.(Note that $V_{3\ell+1}$ in this definition is different from the one in Definition~\ref{def4}).\\
(ii) In $V_{3\ell+1}$, $v_{(3\ell+1)i}$ is adjacent to $v_{(3\ell+1)j}$ for $i\neq j\in [r]$, i.e., the graph induced by $V_{3\ell+1}$ is copy of $K_r$.\\
Let $G^8_{r,k,t}$ be the disjoint union of $t$ copies of $H^8_{r,k}$ (see Figure~\ref{fig:8}).
\end{defn}

\begin{figure}[!ht]\centering
\includegraphics[scale=0.2]{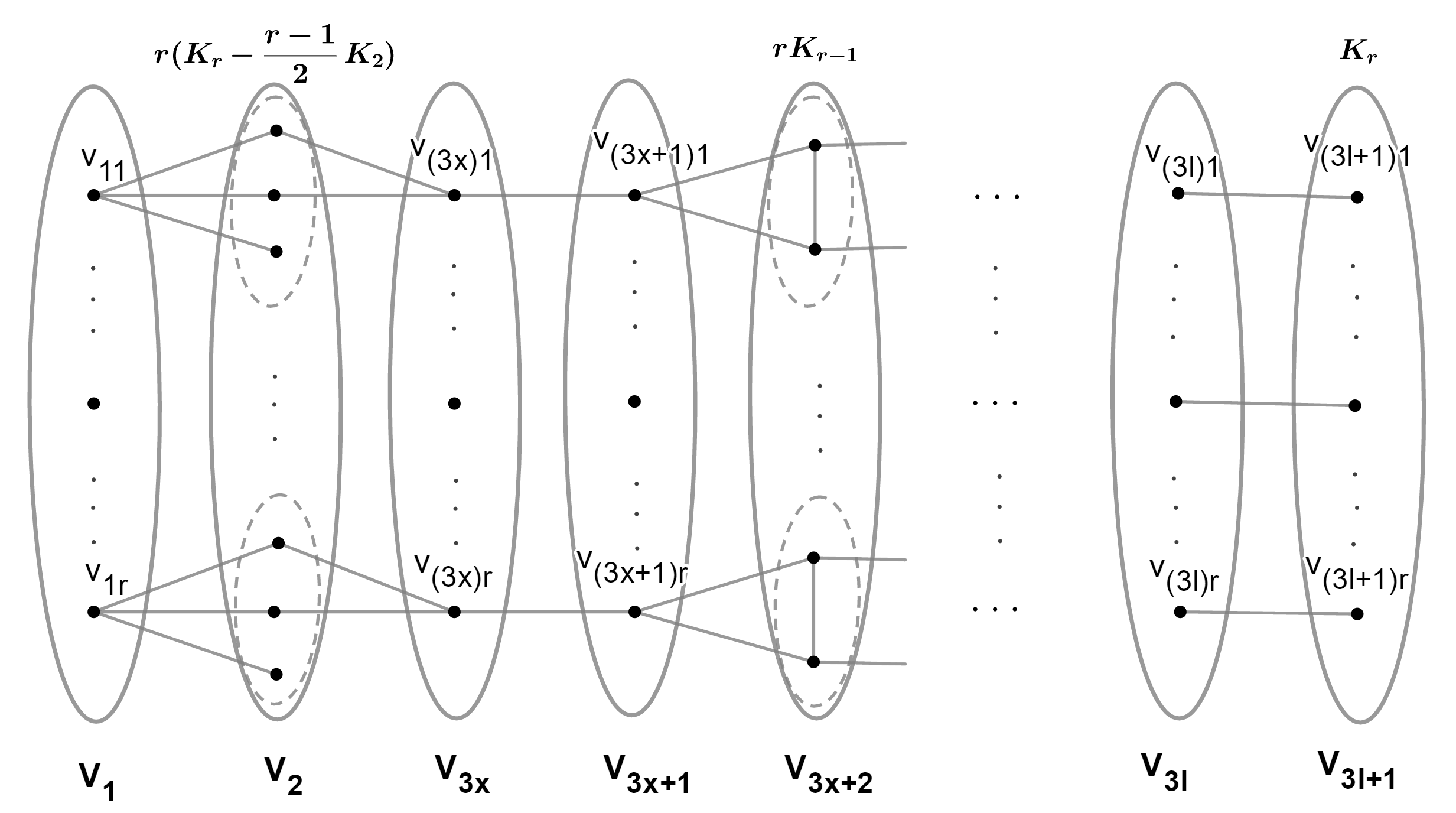}
\caption{The graph $H^8_{r,k}$ \label{fig:8}}
\end{figure}

\begin{defn} \label{def7}
Let $r$ be an even integer at least $4$, and for a positive integer $\ell$, let $k=6\ell$. Let $H^9_{r,k}$ be the $r$-regular graph with the vertex sets $V_1,\ldots,V_{3\ell+1}$ satisfying the following properties:\\
(i) For $x\in [\ell]$, follow the definitions of $V_1,\ V_2,\ V_{3x}, \ V_{3x+1}, \ V_{3x+2}$ in Definition~\ref{def5}, except $|V_1|=\frac{r}{2}$.(Note that $V_{3\ell+1}$ in this definition is different from the one in Definition~\ref{def5}).\\
(ii) In $V_{3\ell+1}$, $v^h_{(3\ell+1)i}$ is adjacent to $v^h_{(3\ell+1)j}$ for $i\neq j\in [r]$ and $h \in \{1,2 \}$, i.e., the graph induced by $V_{3\ell+1}$ is a copy of $K_r$.\\
Let $G^9_{r,k,t}$ be the disjoint union of $t$ copies of $H^9_{r,k}$ (see Figure~\ref{fig:9}).
\end{defn}

\begin{figure}[!ht]\centering
\includegraphics[scale=0.25]{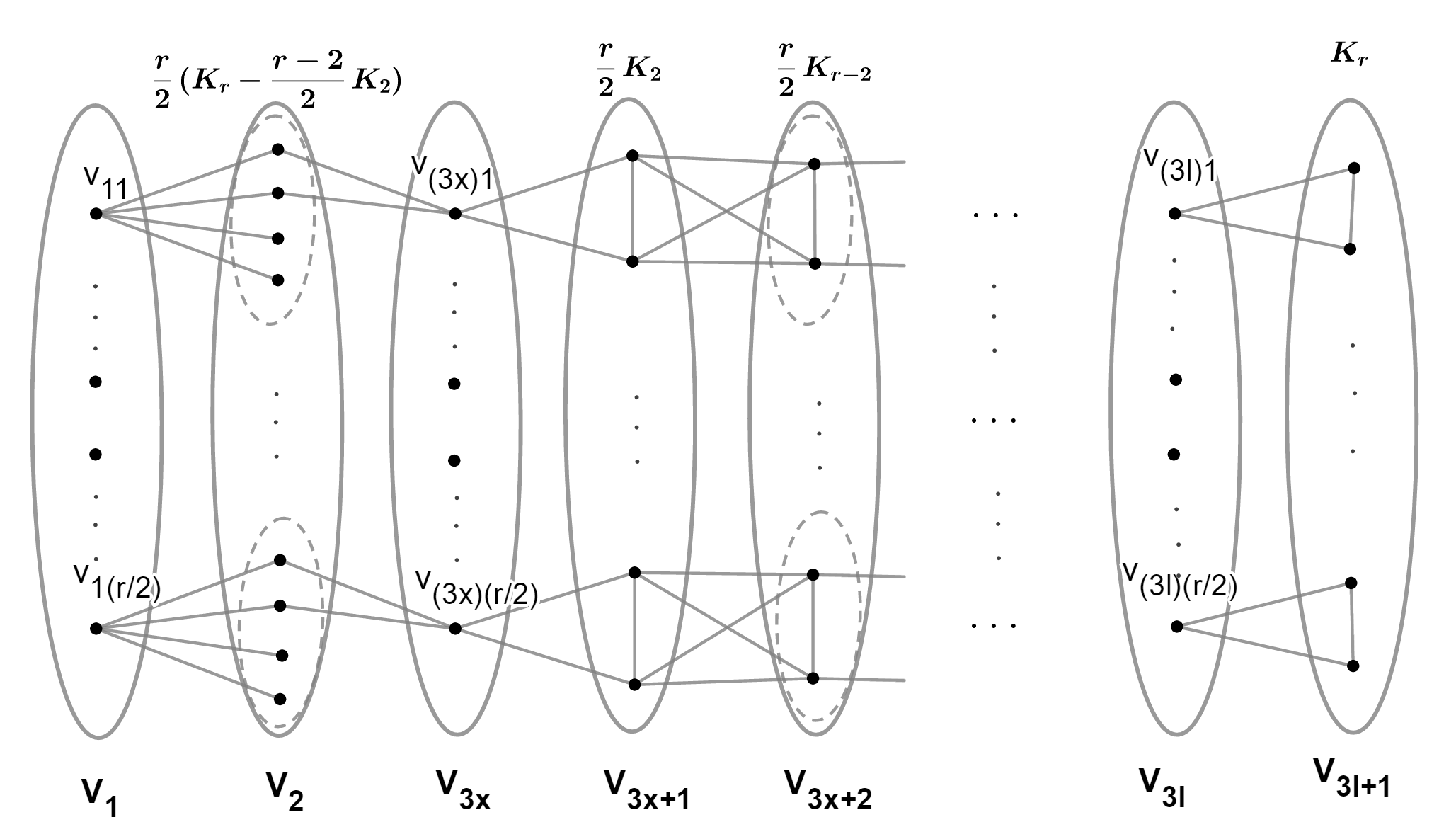}
\caption{The graph $H^9_{r,k}$ \label{fig:9}}
\end{figure}

\begin{obs} \label{obs5}
The graph $G^8_{r,k,t}$ in Definition~\ref{def6} is an $r$-regular graph with $n=t\ell r(r+1)+2tr$ vertices and the $k$-independence number $\frac{n}{\ell (r+1)+2}$.\\
Also, the graph $G^9_{r,k,t}$ in Definition~\ref{def7} is an $r$-regular graph with $\frac {t\ell r(r+1)+3tr}2$ vertices and the $k$-independence number $\frac{n}{\ell (r+1)+3}$.
\end{obs}

\begin{defn} \label{def8}
Let $r$ be an odd integer at least $3$ and for a positive integer $\ell$, let $k=6\ell+1$. Let $H^{10}_{r,k}$ be the $r$-regular graph with the vertex sets $V_1,\ldots,V_{3\ell+2}$ satisfying the following properties:\\
(i) For $x\in [\ell]$, follow the definitions of $V_1,\ V_2,\ V_{3x}, \ V_{3x+1}, \ V_{3x+2}$ in Definition~\ref{def6}, except $V_{3\ell+2}$.(Note that $V_{3\ell+1}$ in this definition is different from the one in Definition~\ref{def6}).\\
(ii) Let $V_{3\ell+2}=\{v_{(3\ell+2)1},\cdots,v_{(3\ell+2)(r-1)} \}$ such that for each $i\in [r-1]$, $v_{(3\ell+2)i}$ is adjacent to all vertices in $V_{3\ell+1}$.\\
(iii) $V_{(3\ell+2)}$ is independent.\\
Let $G^{10}_{r,k,t}$ be the disjoint union of $t$ copies of $H^{10}_{r,k}$ (see Figure~\ref{fig:10}).
\end{defn}

\begin{figure}[!ht]\centering
\includegraphics[scale=0.2]{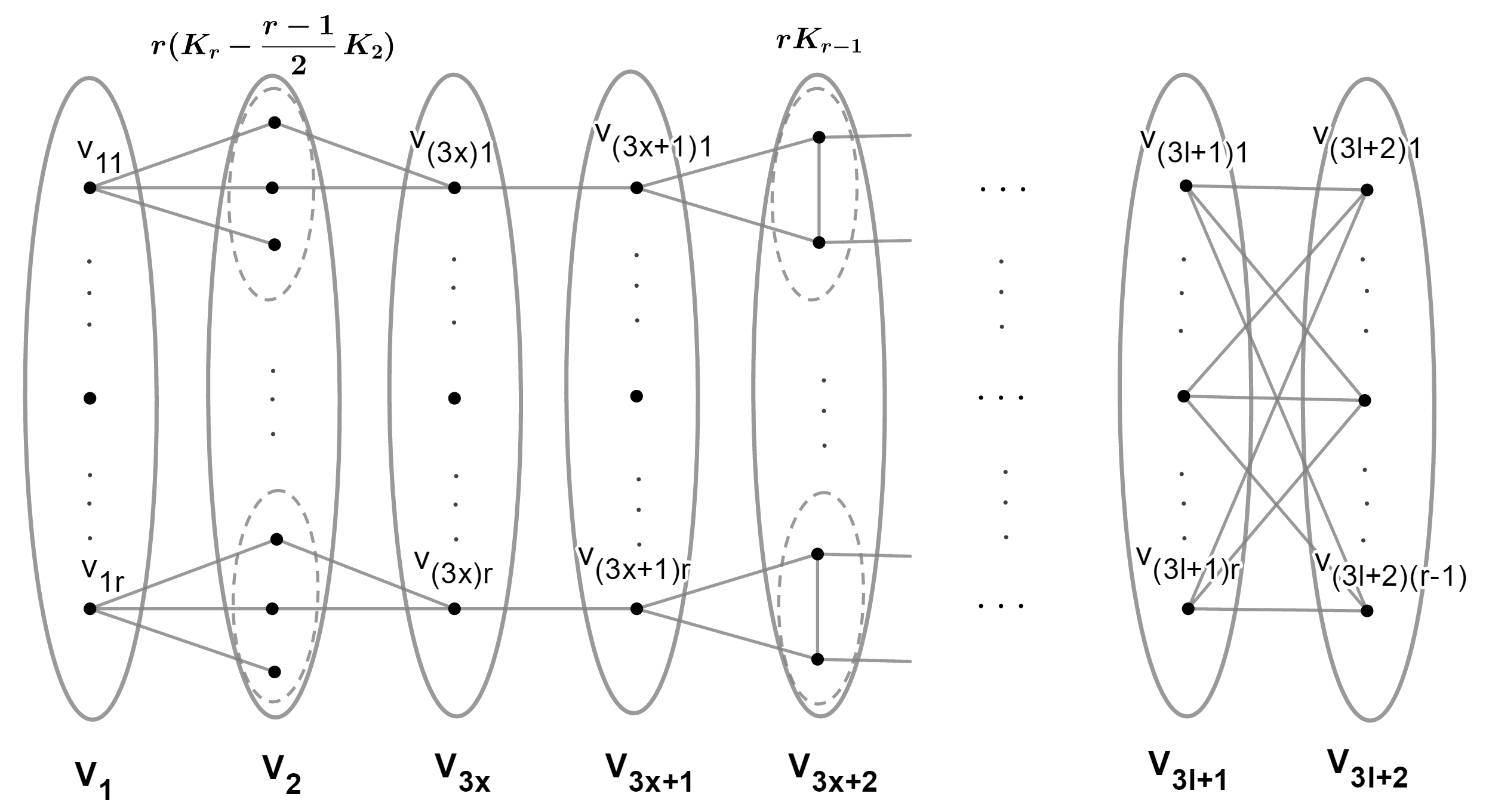}
\caption{The graph $H^{10}_{r,k}$ \label{fig:10}}
\end{figure}

\begin{defn} \label{def9}
Let $r$ be an even integer at least $4$ and for a positive integer $\ell$, let $k=6\ell+1$. Let $H^{11}_{r,k}$ be the $r$-regular graph with the vertex sets $V_1,\ldots,V_{3\ell+2}$ satisfying the following properties:\\
(i) For $x\in [\ell]$, make similiar definitions of $V_1,\ V_2,\ V_{3x}, \ V_{3x+1}, \ V_{3x+2}$ as $H^2_{r,k}$, but change the positions of $V_{3x}$ and $V_{3x+1}$ and make $V_{3\ell+2}$ and $V_2$ a little different.\\
(ii) Let $V_{3\ell+2}= \{v_{(3\ell+2)1}, \cdots, v_{(3\ell+2)(r-2)} \}$ such that for each $i\in [r-2]$, $v_{(3\ell+2)i}$ is adjacent to all vertices in $V_{3\ell+1}$.\\
(iii) $V_{(3\ell+2)}$ is independent.\\
Let $G^{11}_{r,k,t}$ be the disjoint union of $t$ copies of $H^{11}_{r,k}$ (see Figure~\ref{fig:11}).
\end{defn}

\begin{figure}[!ht]\centering
\includegraphics[scale=0.25]{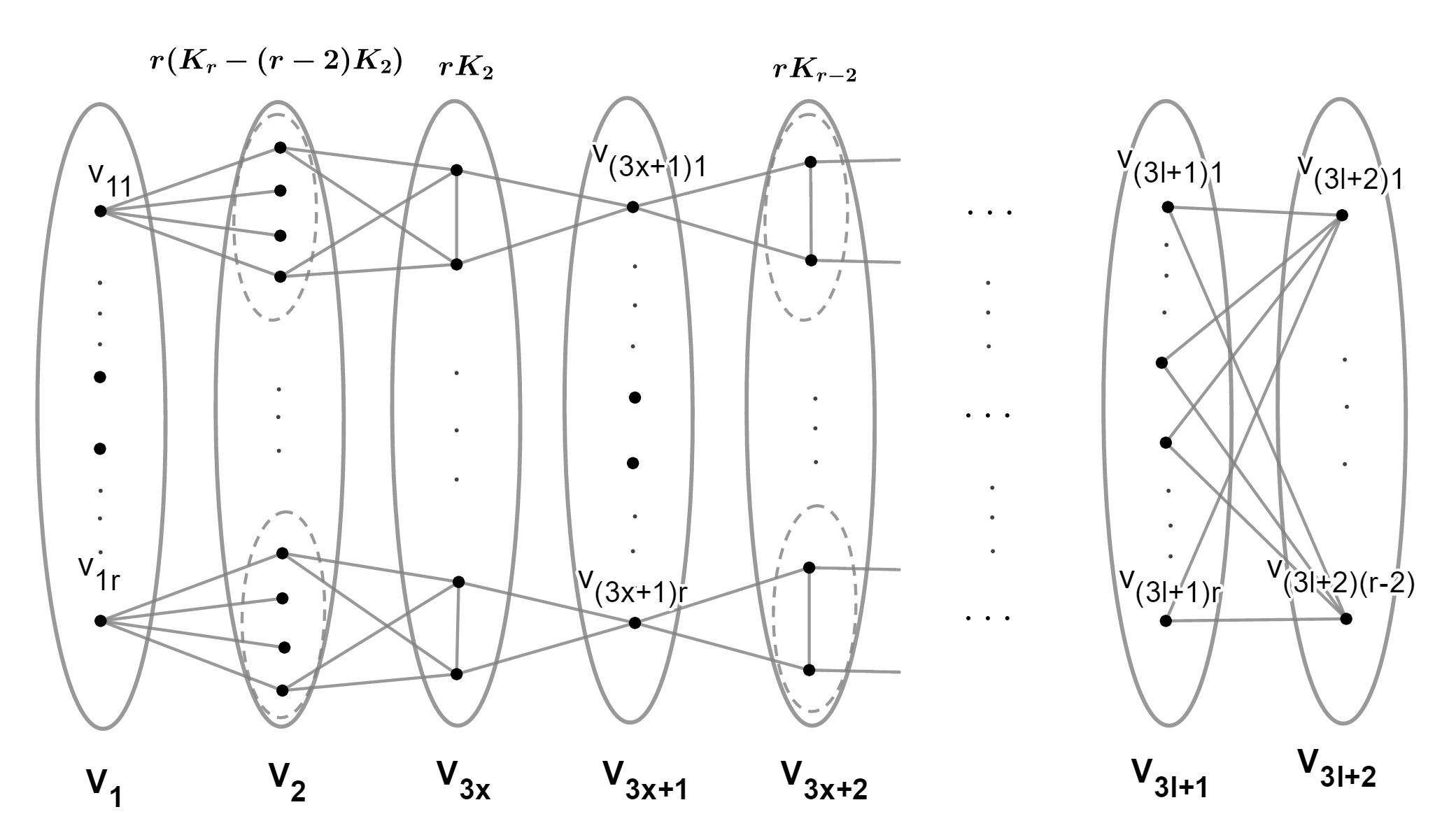}
\caption{The graph $H^{11}_{r,k}$ \label{fig:11}}
\end{figure}

\begin{obs}
The graph $G^{10}_{r,k,t}$ in Definition~\ref{def8} is am $r$-regular graph with $n=t(\ell r+3)(r+1)-4t$ vertices and the $k$-independence number $\frac{rn}{(\ell r+3)(r+1)-4}$.\\
Also, the graph $G^{11}_{r,k,t}$ in Definition~\ref{def9} is an $r$-regular graph with $t(\ell r+4)(r+1)-6t$ vertices and the $k$-independence number $\frac{rn}{(\ell r+4)(r+1)-6}$.
\end{obs}

\section{Sharp Upper Bounds}
In this section, for a positive integer $k$, we prove sharp upper bounds for $\alpha_k(G)$ in an $n$-vertex connected graph $G$ with $diam(G)\ge k+1$. Before proving the bounds, we investigate the relevant properties of a $k$-independent set of $G$. \par
Now, We recall the definition of $N^i(S)$, which is the subsequent neighborhood of $N^{i-1}(S)$, i.e., $N^i(S)=N(N^{i-1}(S))\setminus (N^{i-2}(S)\cup N^{i-1}(S))$. Note that $N^0(S)=S$ and $N^1(S)=N(S)$.\par
For $S\subseteq V(G)$, we denote by $G[S]$ the graph induced by $S$.

\begin{lem} \label{lem}
Let $k$ be a positive integer and let $G$ be an $n$-vertex connected graph with $diam(G)\ge k+1$. Suppose that $S$ is a $k$-independent set in $G$. If $N^{i-1}(S),\ N^i(S),\ N^{i+1}(S)$ are three consecutive sets of $G$ as defined, where $3\leq i \leq \frac{k}{2}-1$, then we have
\begin{equation}\label{lem1}
	|N^{i-1}(S)|+|N^i(S)|+|N^{i+1}(S)|\ge 3|S| ~ \text{ for any} \delta,
\end{equation}
\begin{equation}\label{lem2}
    |N^{i-1}(S)|+|N^i(S)|+|N^{i+1}(S)|\ge (\delta+1)|S| ~\text{ for } \delta\ge 2.
\end{equation}
\end{lem}

\begin{proof}
Let $v_j, v_h \in S$. Note that for $i \in \{0,\ldots,
\lfloor\frac k2\rfloor\}$, we have $N^i(v_j)\cap N^i(v_h)=\emptyset$ and for any $x\in S$, we have $N^i(x) \neq \emptyset$, which implies that $|N^i(S)| \ge |S|$.

Note that for each $u\in N^i(v_j)$, we have $N[u]\subseteq N^{i-1}(v_j)\cup N^i(v_j)\cup N^{i+1}(v_j)$, which implies $|N^{i-1}(v_j)|+|N^i(v_j)|+|N^{i+1}(v_j)|\geq 3$ for any $\delta$.
If $\delta \ge 2$, then
we have $|N^{i-1}(v_j)|+|N^i(v_j)|+|N^{i+1}(v_j)|\geq \delta+1$, which gives the desired result.

\end{proof}

%

Lemma~\ref{lem} is used to prove Theorem~\ref{main2}, which gives upper bounds for $\alpha_k(G)$ in an $n$-vertex connected graph with given minimum and maximum degree.

\begin{thm}\label{main2}
For positive integers $k$ and $\ell$, let $\delta$ and $\Delta$ be the minimum and maximum degree of $G$ respectively. If $G$ is an $n$-vertex connected graph with $diam(G)\ge k+1$, then we have
\begin{enumerate}
	\item If $k=1$, then $\alpha_k(G)\le \frac{\Delta n}{\Delta+\delta}$.
	\item If $k\ge 2$ and $\delta\le 2$, then
	\begin{equation}
		\alpha_k(G)\le
		\begin{cases}
		\frac{\Delta n}{\Delta(\delta+\frac{k-1}{2})+1} & \text{ if }  k \text{ is odd,} \\
		\frac{n}{\delta +\frac{k}{2}} & \text{ if }  k \text{ is even.}
		\end{cases}
	\end{equation}
	\item If $k=6\ell-4$ and $\delta\ge 3$, then $\alpha_k(G)\le \frac{n}{\ell(\delta+1)}$.
	\item If $k=6\ell-3$ and $\delta\ge 3$, then
	\begin{equation}
		\alpha_k(G) \le
		\begin{cases}
		\frac{\Delta n}{\ell \Delta+\ell \delta \Delta+1} & \text{if $\Delta>\delta$,} \\
		\frac{\Delta n}{\ell \Delta+\ell \delta \Delta+2} & \text{if $\Delta=\delta$.}
		\end{cases}
	\end{equation}
	\item If $k=6\ell-2$ and $\delta\ge 3$, then $\alpha_k(G)\le \frac{n}{\ell(\delta+1)+1}$.
	\item If $k=6\ell-1$ and $\delta\ge 3$, then
	\begin{equation}
		\alpha_k(G) \le
		\begin{cases}
		\frac{\Delta n}{\ell \Delta(\delta+1)+\Delta+2} & \text{if $\Delta=\delta$ is even,} \\
		\frac{\Delta n}{\ell \Delta(\delta+1)+\Delta+1} & \text{otherwise.}
		\end{cases}
	\end{equation}
	\item If $k=6\ell$ and $\delta\ge 3$, then
	\begin{equation}
		\alpha_k(G) \le
		\begin{cases}
		\frac{n}{\ell(\delta+1)+3} & \text{if $\Delta=\delta$ is even,} \\
		\frac{n}{\ell(r\delta+1)+2} & \text{otherwise.}
		\end{cases}
	\end{equation}
	\item If $k=6\ell+1$ and $\delta\ge 3$, then
	\begin{equation}
		\alpha_k(G) \le
		\begin{cases}
		\frac{\Delta n}{\ell \Delta(\delta+1)+2\Delta+\delta-1} & \text{if $\delta$ is odd,} \\
		\frac{\Delta n}{\ell \Delta(\delta+1)+3\Delta+\delta-2} & \text{if $\delta$ is even.}
		\end{cases}
	\end{equation}
\end{enumerate}
For $i \in \{1,\ldots, 11\}$, $k\ge 2$, and $\delta\ge 3$, equalities hold for the graphs $G^{i}_{r,k,t}$.
\end{thm}

\begin{proof}
Let $S$ be a $k$-independent set of $G$. Note that $|S|\ge 1$. \\
{\it Case 1: $k=1$.} Note that $|S_1|\delta \le |[S_,\overline{S_1}]| \le \Delta(n-|S_1|)$, where
$[S,T]$ is the set of edges with endpoints in both $S$ and $T$. Thus we have $\alpha_1(G) \le \frac{\Delta n}{\Delta+\delta}$. Equality in the bound requires that $G$ is a $(\delta, \Delta)$-biregular, where a graph is $(a,b)$-biregular if it is bipartite with the vertices of one part all having degree $a$ and the others all having degree $b$.\\
{\it Case 2: $k\ge 2$ and $\delta\le 2$.}
If $k$ is odd, then we have $|N(S)|\ge \delta|S|$ and $|N^i(S)|\ge |S|$, where $i\in \{2,3,\cdots,t-1\}$ and $t=\frac{k+1}{2}$. Since $N^t(u)\cap N^t(v)$ may not be empty for $u,v \in S$, we have $|N^t(S)|\ge \frac{|S|}{\Delta}$. Thus we have $|S|+\delta|S|+(t-2)|S|+\frac{|S|}{\Delta}\le n$, which gives the desired result. If $\delta=1$ and $\Delta=\frac{n-1}{t}$, we have $|S|\le \frac{2n-1}{k+1}$, which gives the bound in Theorem~\ref{thm1}.\\
Similiarly to the proof of odd $k$, for even $k$, we have $|N(S)|\ge \delta|S|$ and $|N^i(S)|\ge |S|$, where $i\in \{2,3,\cdots,t-1\}$ and $t=\frac{k}{2}$. However, $N^t(u)\cap N^t(v)=\emptyset$ for any $u,v \in S$. Thus we have $|N^t(S)|\ge |S|$. Then we have $|S|+\delta|S|+(t-1)|S|\le n$, which gives the desired result. If $\delta=1$, we have $|S|\le \frac{2n}{2+k}$, which gives the bound in Theorem~\ref{thm1}.

From Case 3, we assume that $\delta \ge 3$.\\
{\it Case 3: $k=6\ell-4$.} For any pair of vertices $u, v \in S$, we have $d(u, v)\ge 6\ell-3$.\\
Assume that $u$ and $v$ are two distinct vertices in $S$ with $d(u,v) = 6\ell-3$.
Then there is a path $P=\{u,x_1,\ldots,x_{3\ell-2},y_{3\ell-2},\ldots,y_1,v \}$ with legnth $6\ell-3$. Note that since $d(u,v)=6\ell-3$, which is odd, we have $N^{3\ell-3}(u)\cap N^{3\ell-3}(v) =\emptyset$ and there are edges between $N^{3\ell-2}(u)$ and $N^{3\ell-2}(v)$.\\
Note that $|N^1(S)|\ge \delta|S|$ and $S$ is $k$-independent. For a positive integer $h \in [\ell-1]$, by considering $N^{3h-1}(S), N^{3h}(S), N^{3h+1}(S)$ as a unit, we have at least $(\ell-1)$ units since $S$ is a $(6\ell-4)$-independent set.\\
Thus by Lemma~\ref{lem}~(\ref{lem2}), we have $|S|+\delta|S|+(\ell-1)(\delta+1)|S| \le n$, which gives the desired result. Equality holds for the graphs $G^i_{r,k,t}$ for all $i \in \{1,2,3\}$ when $\delta=r$.\\
{\it Case 4: $k=6\ell-3$.} The proof is similar to that of Case 3. Since there are two vertices $u$ and $v$ in $S$ such that $d(u,v)=6\ell-2$, there is a path $P=\{u,x_1,\ldots,x_{3\ell-2},z,y_{3\ell-2},\ldots,y_1,v \}$ with legnth $6\ell-2$. Note that $N^{3\ell-1}(u) \cap N^{3\ell-1}(v)$ can be non-empty.\\
Since there are $(\ell-1)$ units and $|N^{3\ell-1}(S)| \ge \frac{|S|}{\Delta}$ for $\Delta>\delta$, we have $|S|+\delta|S|+(\ell-1)(\delta+1)|S|+\frac{|S|}{\Delta} \le n$. Since $|N^{3\ell-1}(S)| \ge \frac{2|S|}{\Delta}$ for $\Delta=\delta$, we have $|S|+\delta|S|+(\ell-1)(\delta+1)|S|+\frac{2|S|}{\Delta} \le n$, which gives the desired results. Equality holds for the graph $G^4_{r,k,t}$ when $\delta=\Delta=r$.\\
{\it Case 5: $k=6\ell-2$.} In this case, we consider $N^{3h}(S), N^{3h+1}(S), N^{3h+2}(S)$ as a unit. Then by Lemma~\ref{lem}~(\ref{lem2}), we have $|S|+\delta|S|+|S|+(\ell-1)(\delta+1)|S| \le n$ since $|N^2(S)|\ge |S|$. Equality holds for the graph $G^5_{r,k,t}$ when $\delta=r$.\\
{\it Case 6: $k=6\ell-1$.} Similarly to Case 5, we consider $N^{3h}(S), N^{3h+1}(S), N^{3h+2}(S)$ as a unit. Since $N^{3\ell}(u) \cap N^{3\ell}(v)$ can be non-empty for two vertices $u$ and $v$ with $d(u,v)=6\ell$, we have $|S|+\delta|S|+|S|+(\ell -1)(\delta+1)|S|+\frac{2|S|}{\Delta}\leq n$ for even $\Delta=\delta$ and we have $|S|+\delta|S|+|S|+(\ell -1)(\delta+1)|S|+\frac{|S|}{\Delta}\leq n$ for odd $\Delta=\delta$ or $\Delta>\delta$. Equalities hold for the graphs $G^6_{r,k,t}$ and $G^7_{r,k,t}$ for $\delta=\Delta=r$ depending on the parity of $r$.\\
{\it Case 7: $k=6\ell$.} In this case, we have $N^{3\ell}(u) \cap N^{3\ell}(v) = \emptyset$ for two vertices $u$ and $v$ with $d(u,v)=6\ell+1$. Thus for even $\Delta=\delta$, we have $|N^{3\ell}(S)|\ge 2|S|$, which implies $|S|+\delta|S|+|S|+(\ell -1)(\delta+1)|S|+2|S|\leq n$, and for odd $\Delta=\delta$ or $\Delta>\delta$, we have $|N^{3\ell}(S)|\ge |S|$, which implies $|S|+\delta|S|+|S|+(\ell -1)(\delta+1)|S|+|S|\leq n$. Equalties hold for the graphs $G^8_{r,k,t}$ and $G^9_{r,k,t}$ when $\delta=\Delta=r$ depending on the parity of $r$.\\
{\it Case 8: $k=6\ell+1$.} Like Case 3, we consider $N^{3h-1}(S), N^{3h}(S), N^{3h+1}(S)$ as a unit. Note that $N^{3\ell+1}(u) \cap N^{3\ell+1}(v)$ can be non-empty for two vertices $u$ and $v$ with $d(u,v)=6\ell+2$. Thus for odd $\delta$, we have $|N^{3\ell-1}(S)|\ge |S|$, $|N^{3\ell}(S)|\ge |S|$ and $|N^{3\ell+1}(S)|\ge \frac{(\delta-1)|S|}{\Delta}$, which implies $|S|+\delta|S|+(\ell-1)(\delta+1)|S|+|S|+|S|+\frac{(\delta-1)|S|}{\Delta}\leq n$,
and for even $\delta$, we have $|N^{3\ell-1}(S)|\ge 2|S|$, $|N^{3\ell}(S)|\ge |S|$ and $|N^{3\ell+1}(S)|\ge \frac{(\delta-2)|S|}{\Delta}$, which implies $|S|+\delta|S|+(\ell-1)(\delta+1)|S|+2|S|+|S|+\frac{(\delta-2)|S|}{\Delta}\leq n$. Equalties hold for the graphs $G^{10}_{r,k,t}$ and $G^{11}_{r,k,t}$ when $\delta=\Delta=r$ depending on the parity of $r$.
\end{proof}

\section{Questions}
Aida, Cioab{\'a}, and Tait~\cite{ACT} obtained two spectral upper bounds for the k-independence number of a graph. They constructed graphs that attain equality for their first bound and showed that their second
bound compares favorably to previous bounds on the k-independence number. We may ask whether given an independence number, there is an upper or lower bound for the spectral radius (the largest eigenvalue of a graph) in an $n$-vertex regular graph.

\begin{ques}
Given a positive integer $t$, what is the best lower bound for the spectral radius in an $n$-vertex $r$-regular graph to guarantee that $\alpha_k(G) \ge t+1$?
\end{ques}

If for $r \ge 3$, $G$ is an $n$-vertex $r$-regular graph, which is not a complete graph,
then $\alpha_1(G) \ge \frac n{\chi(G)} \ge \frac nr$ by Brooks' Theorem.  For $k \ge 2$, it is natural to ask a lower bound for $\alpha_k(G)$ in an $n$-vertex $r$-regular graph.

\begin{ques}\label{q2}
For $r \ge 3$, what is the best lower bound for $\alpha_k(G)$ in an $n$-vertex $r$-regular graph?
\end{ques}

The $k$-th power of the graph $G$, denoted by $G^k$, is a graph on the same vertex set as $G$ such that two vertices are adjacent in $G^k$ if and only if their distance in $G$ is at most $k$. The $k$-distance $t$-coloring, also called distance $(k, t)$-coloring, is a $k$-coloring of the graph $G^k$ (that is, any two vertices within distance k in G receive different colors).
The $k$-distance chromatic number of G, written $\chi_k(G)$, is exactly the chromatic number of $G^k$. It is easy to see that $\chi(G) = \chi_1(G) \le\chi_k(G) = \chi(G^k)$.

It was noted by Skupie{\' n} that the well-known Brooks' theorem can provide the following upper bound:
\begin{equation}\label{eq2}
\chi_k(G) \le 1 + \Delta(G^k) \le 1 + \Delta \sum_{i=1}^k (\Delta -1)^{k-1} = 1 + \Delta\frac{(\Delta - 1)^k - 1}{\Delta-2},
\end{equation}
for $\Delta \ge 3$.  Let $M=:1 + \Delta\frac{(\Delta - 1)^k - 1}{\Delta-2}$. Consider a $(k, \chi_k(G))$-coloring.  Let $V_i$ be the vertex set with the color $i$ for $i \in [\chi_k(G)]$. Then we have $\chi_k(G)\alpha_k(G) \ge n$. Thus for $r \ge 3$, if $G$ is an $n$-vertex $r$-regular graph,
then we have $\alpha_k(G) \ge \frac n{\chi_k(G)} \ge \frac nM$. Since equality in inequality~(\ref{eq2}) holds only when $G$ is a Moore graph, the lower bound is not tight. Thus, we might be interested in answering Question~\ref{q2}.

\paragraph{Acknowledgments.} We would like to thank the editor and the referee for their careful
reading and helpful comments.

\end{document}